\documentclass[aps,prd,groupedaddress,preprint,tightenlines,eqsecnum,nofootinbib]{revtex4}
\usepackage{graphicx,epsf,amssymb,amsbsy,amsfonts,amssymb,amsmath,amsthm}
\usepackage{hyperref}
\usepackage{mathrsfs}
\usepackage[dvips]{color}
\newcommand{\eq}{\begin{equation}}
\newcommand{\eqe}{\end{equation}}
\newcommand{\g}{\gamma}
\newcommand{\e}{\epsilon}
\renewcommand{\d}{\delta}

\usepackage{setspace}

\newcommand{\eqa}{\begin{eqnarray}}
\newcommand{\eqae}{\end{eqnarray}}

\newcommand{\comment}[1]{}

\newcommand{\SL}{\operatorname{\textsl{SL}}}      %SL group
\newcommand{\GL}{\operatorname{\textsl{GL}}}      %SL group
     
\newcommand{\RR}{{\mathbb R}}%Reals
\newcommand{\C}{{\mathbb C}}%Complex
\newcommand{\CC}{{\mathbb C}}%Complex
\newcommand{\Z}{{\mathbb Z}}%Integers
\newcommand{\ZZ}{{\mathbb Z}}%Integers
\newcommand{\QQ}{{\mathbb Q}}%Rationals
\newcommand{\HH}{{\mathbb H}}%Upper half-plane

\newcommand{\Res}{\operatorname{Res}}
\newcommand{\Ind}{\operatorname{Ind}}

\newcommand{\BM}{\left(\begin{matrix}}
\newcommand{\EM}{\end{matrix}\right)}

\newcommand{\WT}{\widetilde} 
 
\newcommand{\A}{\alpha} 
\newcommand{\B}{\beta} 

\newcommand{\OO}{\mathcal{O}}

\newtheorem{thm}{Theorem}[section]

\newtheorem{lem}[thm]{Lemma}
\newtheorem{prop}[thm]{Proposition}

\begin{document}

\setstretch{1.2}

\titlepage

\title{Modular Forms on the Double Half-Plane}

\author{John F.~R.~Duncan$^1$ and David A.~McGady$^2$}
\affiliation{$^1$Department of Mathematics,
Emory University, Atlanta, GA 30322, USA (\href{mailto:john.duncan@emory.edu}{john.duncan@emory.edu})\\
$^2$Nordita, KTH Royal Institute of Technology and Stockholm University, 
Roslagstullsbacken 23, SE-106 91 Stockholm, Sweden\\
{\rm and} The Niels Bohr International Academy and Discovery Center,\\ Niels Bohr Institute, Blegdamsvej 17, DK-2100 Copenhagen, Denmark (\href{mailto:dmcgady@alumni.princeton.edu}{dmcgady@alumni.princeton.edu})}

\begin{abstract}
We formulate a notion of modular form on the double half-plane for half-integral weights and explain its relationship to the usual notion of modular form. The construction we provide is compatible with certain physical considerations due to the second author. 
\end{abstract}

\maketitle

%\tableofcontents

%---------------------------------------------------------------------------%
\section{Introduction}\label{sec:intro}
%---------------------------------------------------------------------------%

A modular form is a holomorphic function on the complex upper half-plane that enjoys rich symmetry. More specifically, a modular form should transform in a prescribed way with respect to the mappings that encode conformally equivalent two-dimensional tori (see e.g. \cite{01-Apostol, 02-Serre, 03-Knapp} for detailed expositions). For this reason two-dimensional conformal field theory is a rich source of modular forms (see e.g. \cite{04-Yellow, 05-PolchinskiI}), but in \cite{06T-rex2} it is argued that the modular forms that arise in this way should be defined on both the upper and lower half-planes, and moreover should transform nicely with respect to the reflection $z\mapsto -z$, which interchanges these two half-planes. 

The argument for modular forms of even weight (with trivial character) is particularly transparent: these forms constitute a ring which is generated by the {\em Eisenstein series} $E_4$ and $E_6$, where
\begin{gather}\label{eqn:Ekdefn}
 E_k(z):=\sum_{\substack{m,n\in\ZZ\\(m,n)\neq (0,0)}} (mz+n)^{-k}.
\end{gather}
The right hand side of (\ref{eqn:Ekdefn}) is exactly the same if we replace $z$ with $-z$ so it is natural to extend $E_k$, and any modular form of even weight, to the {\em double half-plane} 
\begin{gather}
\HH^*:=\{z\in \CC\mid \Im(z)\neq 0\}
\end{gather} 
by requiring $f(z)=f(-z)$ for $z\in \HH^*$. 

What can we say about more general weights and more general multipliers? For a motivating example consider the {\em Dedekind eta function}, which is defined for $\Im(z)>0$ by setting 
\begin{gather}
\eta(z):=e^{\frac{\pi i z}{12}}\prod_{n=1}^\infty(1-e^{2\pi i n z}).
\end{gather} 
This function satisfies $\eta(z+1)=e^{\frac{\pi i}{12}}\eta(z)$ and $\eta(-\frac1z)=\sqrt{-iz}\eta(z)$ for $\Im(z)>0$ where $\sqrt{re^{i\theta}}:=\sqrt{r}e^{i\frac{\theta}{2}}$ for $-\pi<\theta<\pi$. It is a modular form of weight $\frac12$ with a certain multiplier, and is one of the many examples that arise from conformal field theory. 

If the Dedekind eta function is to extend to the lower half-plane, how should its values at $z$ and $-z$ be related? The following heuristic answer is presented in \cite{07T-rex1, 06T-rex2}. Define $Z_N(z):=\prod_{n=1}^N(e^{-\pi i n z}-e^{\pi i n z})$ for $z\in \CC$ and recall that the {\em Riemann zeta function} $\zeta(s)$ is a meromorphic function on the complex plane that satisfies $\zeta(s)=\sum_{n=1}^\infty n^{-s}$ for $\Re(s)>1$, and also $\zeta(0)=-\frac12$ and $\zeta(-1)=-\frac1{12}$. We have
\begin{gather}
	Z_N(z)=(e^{-\pi i z})^{\sum_{n=1}^N n}\prod_{n=1}^N(1-e^{2\pi i n z}).
\end{gather}
So we obtain a formal identity ``$\lim_{N\to \infty}Z_N(z)=\eta(z)$'' if we replace $\lim_{N\to \infty}\sum_{n=1}^N n$ with $\zeta(-1)=-\frac1{12}$ in $\lim_{N\to \infty}Z_N(z)$. Now we also have $Z_N(-z)=(-1)^NZ_N(z)$. So again taking a limit over $N$ we obtain 
\begin{gather}\label{eqn:etaRxform}
``\eta(-z)=(-1)^{-\frac12}\eta(z)"
\end{gather} 
if we replace $\lim_{N\to \infty}Z_N(z)$ with $\eta(z)$ and $\lim_{N\to \infty}N=\lim_{N\to \infty}\sum_{n=1}^N 1$ with $\zeta(0)=-\frac12$.

So the prediction of \cite{07T-rex1, 06T-rex2, 08T-rex0} is that there should be an extension of $\eta$ to $\HH^*$ with the property that $\eta(z)$ differs from $\eta(-z)$ by a fourth root of unity. (See \cite{09-Lfunctions} for a more general discussion along similar lines.) In this note we show that such an extension exists, and we situate it within the larger framework of {\em (vector-valued) modular forms on the double half-plane}. We formulate this notion for arbitrary integral and half-integral weights (see \S\ref{sec:mdlfmsgl2}), and we characterize all such modular forms in terms of the usual modular forms on the upper half-plane (see Theorems \ref{thm:resisom} and \ref{thm:indisom}). The Dedekind eta function gives rise to a particular example which manifests a precise interpretation of the heuristic identity (\ref{eqn:etaRxform}). 

The rest of the paper is structured as follows. We first review modular forms on the upper half-plane very briefly in \S\ref{sec:modfmssl2}. Then in \S\ref{sec:metcvrgl2} we discuss a certain double cover of $\GL_2(\ZZ)$ which plays a key role in our construction. In \S\ref{sec:fnsdblhlfpln} we define actions of this double cover on holomorphic functions on the double half-plane. In \S\ref{sec:mdlfmsgl2} we define modular forms on the double half-plane and relate them to the usual modular forms. Finally, in \S\ref{sec:egs} we review the examples discussed above from within the new framework.

All of the arguments in this note are elementary, but some are a bit involved. For the sake of completeness we have opted to include more detail rather than less.

%---------------------------------------------------------------------------%
\section{Modular Forms on the Upper Half-Plane}
\label{sec:modfmssl2}
%---------------------------------------------------------------------------%

The {metaplectic double cover} of $\SL_2(\Z)$ may be realized as the set of pairs $(\g,\phi)$ where $\g\in \SL_2(\Z)$ and $\phi$ is a holomorphic function on the upper half-plane $\HH^+ := \{ z \in \C \mid \Im(z) > 0\}$ such that $\phi(z)^2=cz+d$ when $\g=\left(\begin{smallmatrix}*&*\\c&d\end{smallmatrix}\right)$. It becomes a group, which we denote $\widetilde{\SL_2}(\Z)$, when equipped with the multiplication rule 
\begin{align}
\label{eqn:tildesl2mult}
(\A,\phi(z))(\B,\psi(z)):=(\A\B,\phi(\B z)\psi(z)).
\end{align} 

A main application of this formulation is to modular forms of half-integral weight. Indeed, for $k\in \frac12\Z$ the {\em weight $k$ (right) action} of $\widetilde{\SL_2}(\Z)$ on the space $\OO(\HH^+)$ of holomorphic functions on the upper half-plane is given by
\begin{align}
\label{eqn:fplustildegamma}
	(f|_k(\g,\phi))(z) := f(\g z)\phi(z)^{-2k}
\end{align}
(cf. Proposition \ref{prop:itisanaction}) for $f\in \OO(\HH^+)$ and $(\g,\phi)\in \widetilde{\SL_2}(\Z)$ and $z\in \HH^+$. 
Given a homomorphism $\rho:\widetilde{\SL_2}(\Z)\to \C^*$, we may define a {\em modular form} of weight $k$ for $\widetilde{\SL_2}(\Z)$ with character $\rho$ to be a holomorphic function $f\in\OO(\HH^+)$ such that $f|_k\WT\g=\rho(\WT\g)f$ for all $\WT\g\in \widetilde{\SL_2}(\Z)$. 
Often one simply speaks of {\em modular forms for $\SL_2(\Z)$}. For $k\in \Z$ there is no abuse in this since $\phi^{-2k}=(-\phi)^{-2k}$. For $k\in \Z+\frac12$ it's more natural to work with $\widetilde{\SL_2}(\Z)$.

In this note we wish to consider modular forms of (integral and) half-integral weight for $\GL_2(\Z)$. Define the {\em lower half-plane} to be $\HH^-:=\{z\in \C\mid \Im(z)<0\}$. The group $\GL_2(\Z)$ acts naturally on the {double half-plane} $\HH^*=\HH^+\cup\HH^-$, so we require a double cover of $\GL_2(\Z)$ that contains $\widetilde{\SL_2}(\Z)$, and a compatible action on the space $\OO(\HH^*)$ of holomorphic functions on the double half-plane. 

%---------------------------------------------------------------------------%
\section{A metaplectic cover}
\label{sec:metcvrgl2}
%---------------------------------------------------------------------------%

There are exactly two double covers of $\GL_2(\Z)$ that contain a copy of $\widetilde{\SL_2}(\Z)$. The one that will be of use to us is distinguished by the property that any preimage of $R:=\left(\begin{smallmatrix}-1&0\\0&1\end{smallmatrix}\right)\in \GL_2(\Z)$ has order $4$. (In the other double cover that contains $\widetilde{\SL_2}(\Z)$ a preimage of $R$ has order $2$.) 

To perform explicit computations in this cover it is convenient to employ a cocycle due to Kubota~\cite{10-Kubota}. In order to describe this we first introduce Kubota's 
function $\chi:\GL_2(\Z)\to \Z$, which is defined by setting
\begin{align}
\label{eqn:Kubotachi}
\chi(\g) = \chi\left(\left(\begin{smallmatrix} a&b\\c&d\end{smallmatrix}\right)\right):=
\begin{cases}
	c&\text{ if $c\neq 0$,}\\
	d&\text{ if $c=0$.}
\end{cases}
\end{align}
Next we recall the {\em Hilbert symbol} (at the infinite place of $\QQ$) which is defined for $a,b\in \RR^*$ by setting
\begin{align}
\label{eqn:Hilbertsymbol}
(a,b):=
\begin{cases}
	-1&\text{ if $a<0$ and $b<0$,}\\
	1&\text{ otherwise.}
\end{cases}
\end{align}
Then, following~\cite{11-Budden-Goehle}, we define {\em Kubota's twisted $2$-cocycle} for $\GL_2(\Z)$ by setting
\begin{align}
\label{eqAtwist}
A(\A,\B):=
\left(\det(\A),\det(\B)\right)
\left(\frac{\chi(\A\B)}{\chi(\A)},\frac{\chi(\A\B)}{\chi(\B)\det(\A)}\right)
\end{align}
for $\A,\B\in\GL_2(\Z)$. (The cocycle $A$ in (\ref{eqAtwist}) is $\sigma_2^{(1)}$ in the notation of~\cite{11-Budden-Goehle}. The other double cover of $\GL_2(\Z)$ is obtained by using $\sigma_2^{(0)}(\A,\B):= \left(\frac{\chi(\A\B)}{\chi(\A)},\frac{\chi(\A\B)}{\chi(\B)\det(\A)}\right)$ instead.)

Now we may realize $\widetilde{\GL_2}(\Z)$ as the set of pairs $[\g,\e]$ where $\g \in \GL_2(\Z)$ and $\e\in\{\pm 1\}$, and the multiplication is given by
\begin{align}
\label{eqProd}
	[\A,\e][\B,\delta]:=[\A\B,A(\A,\B)\e\delta].
\end{align}
The associativity of the multiplication rule (\ref{eqProd}) is equivalent to the cocycle identity 
\begin{gather}\label{eqn:cocycleidentityA}
	A(\A,\B)A(\A\B,\g)=A(\A,\B\g)A(\B,\g)
\end{gather}
for $\A,\B,\g\in \GL_2(\Z)$.

If we define $S:=\left(\begin{smallmatrix} 0&-1\\1&0\end{smallmatrix}\right)$ and $T:=\left(\begin{smallmatrix} 1&1\\0&1\end{smallmatrix}\right)$ then $\{ S,T\}$ is a generating set for $\SL_2(\Z)$. Define preimages of $S$ and $T$ in $\widetilde{\GL_2}(\Z)$ by setting $\WT S:= [S,1]$ and $\WT T := [T,1]$, and also set $\WT Z:=[-I,1]$. Then $\WT S$ and $\WT T$ generate a subgroup of $\widetilde{\GL_2}(\Z)$ that is isomorphic to $\widetilde{\SL_2}(\Z)$, and the center of this copy of $\widetilde{\SL_2}(\Z)$ is generated by $\WT S^2=\WT Z$, which has order $4$. Next define $R:=\left(\begin{smallmatrix} -1&0\\0&1\end{smallmatrix}\right)$ and set $\WT R:=[R,1]$. Then, with the twisted product defined in Eqs.~\eqref{eqAtwist} and~\eqref{eqProd}, we have that $\WT S^4=\WT Z^2=\WT R^2 =[I,-1]$ is the unique non-trivial central element of $\widetilde{\GL_2}(\Z)$. We pause here to emphasize that whilst $\WT Z$ is a generator for the center of $\widetilde{\SL_2}(\Z)$, which has order $4$, it will develop momentarily that there is a unique non-trivial central element in $\widetilde{\GL_2}(\Z)$. In particular $\WT Z$ is not central in $\widetilde{\GL_2}(\Z)$.

The fact that $\chi(R\gamma)=\chi(\gamma)$ for $\gamma\in \GL_2(\Z)$ leads to useful identities for the cocycle (\ref{eqAtwist}). The next lemma is a basic example of this. To formulate it set
\begin{align}
\label{eqn:betagamma}
	B(\g):=(\chi(\g),\chi(\g R))
\end{align}
for $\gamma\in \SL_2(\Z)$.
\begin{lem}\label{lem:AAB}
For $\g\in\SL_2(\Z)$ we have $A(R,\gamma)A(R\gamma,R)=A(R,\gamma R)A(\gamma,R) = (-1)B(\gamma)$.
\end{lem}

\begin{proof}
The first equality is a special case of (\ref{eqn:cocycleidentityA}) so we just check that $A(R,\g)A(R\g,R)=(-1)B(\g)$.
Using $\chi(R\gamma)=\chi(\gamma)$ and the hypothesis that $\det(\g)=1$ we have 
\begin{gather}
\begin{split}
A(R,\g)A(R\g,R)=&
(-1,1)\left(\frac{\chi(R\g)}{\chi(R)},\frac{\chi(R\gamma)}{\chi(\g)\det (R)}\right)\\
& \times
(-1,-1)\left(\frac{\chi(R\g R)}{\chi(R\g)},\frac{\chi(R\g R)}{\chi(R)\det (R\g)} \right)\\
=&
(-1)\left({\chi(\g)},-1\right)
\left({\chi(\g)}{\chi(\g R)},-{\chi(\g R)} \right).
\end{split}
\end{gather}
So we require to verify that 
\begin{align}
\label{eqn:chigammachiRgammaR}
(\chi(\g),-1)({\chi(\g)} {\chi(\g R)},-\chi(\g R))
=
(\chi(\g),\chi(\g R)).
\end{align}
Observe that if $\g=\left(\begin{smallmatrix}*&*\\c&d\end{smallmatrix}\right)$ then $\g R=\left(\begin{smallmatrix}*&*\\-c&d\end{smallmatrix}\right)$. So if $c\neq 0$ then the identity~\eqref{eqn:chigammachiRgammaR} becomes $(c,-1)(-c^2,c)=(c,-c)$ which is true since $(-c^2,c)=(c,-1)$ and $(c,-c)=1$. If $c=0$ then~\eqref{eqn:chigammachiRgammaR} becomes $(d,-1)(d^2,-d)=(d,d)$ which is true because $(d^2,-d)=1$ and $(d,-1)=(d,d)$. This proves the lemma.
\end{proof}

\begin{lem}
\label{lem:RgammaR}
For $\g\in {\SL_2}(\Z)$ and $\e\in\{\pm1\}$ we have $\WT R[\g,\e]\WT R^{-1} = [R\g R, B(\g)\e]$.
\end{lem}

\begin{proof}
We have
\begin{gather}
\begin{split}
\WT R[\g,\e]\WT R^{-1}&=[R,1][\g,\e][R,-1]\\
&=[R\g,A(R,\g)\e][R,-1]\\
&=[R\g R,A(R,\g)A(R\g,R)(-\e)].
\end{split}
\end{gather}
So the claimed identity follows from Lemma \ref{lem:AAB}.
\end{proof}

\begin{lem}
\label{LemI2}
The element $\WT R\in\widetilde{\GL_2}(\Z)$ conjugates any preimage of $S$ or $T$ to its inverse in $\widetilde{\GL_2}(\Z)$.
\end{lem}
\begin{proof}
We have $B(S)=(\chi(S),\chi(RSR))=(1,-1)=1$ so $\WT R[ S,1]\WT R^{-1}=[RSR,1]=[-S,1]$ by Lemma \ref{lem:RgammaR} and $[ S,1]^{-1}=[-S,1]$ because $[S,1][-S,1]=[I,A(S,-S)]=[I,1]$. To verify that $\WT R [ T,1]\WT R^{-1}=[ T,1]^{-1}$ we compute $B(T)=(\chi(T),\chi(RTR))=(1,1)=1$, which gives $\WT R[ T,1]\WT R^{-1}=[RTR,1]$ according to Lemma \ref{lem:RgammaR}. Then we verify that $[T,1][RTR,1]=[I,A(T,RTR)]=[I,1]$.
\end{proof}
Lemma \ref{LemI2} shows that $\WT R$ conjugates $\WT S^2=\WT Z$ to its inverse. Since $\WT Z$ is not self-inverse this confirms that $\WT Z$ is not central in $\widetilde{\GL_2}(\Z)$.

%---------------------------------------------------------------------------%
\section{Functions on the Double Half-Plane}
\label{sec:fnsdblhlfpln}
%---------------------------------------------------------------------------%

To formulate modular forms on the double half-plane we require an action of $\widetilde{\GL_2}(\Z)$ on $\OO(\HH^*)$. To achieve this we first extend the action of $\widetilde{\SL_2}(\Z)$ to $\OO(\HH^*)$ as follows. Choose an embedding $\iota:\widetilde{\SL_2}(\Z)\to\widetilde{\GL_2}(\Z)$ that lifts the natural embedding $\SL_2(\Z)\to \GL_2(\Z)$ (i.e. such that $\iota(\A,\phi)=[\B,\e]$ implies $\A=\B$). Then for each $\g\in \SL_2(\Z)$ let $\phi_{\g}^+$ be the function $\HH^+\to\C$ such that $(\g,\phi_{\g}^+)$ is the preimage of $[\g,1]$ under $\iota$. The multiplication rule~\eqref{eqProd} implies $(\A,\phi^+_{\A})(\B,\phi^+_{\B})=(\A\B,A(\A,\B)\phi^+_{\A\B})$, so by~\eqref{eqn:tildesl2mult} we have
\begin{align}
\label{eqP1}
	\phi^+_\A(\B z)\phi^+_{\B}(z) = A(\A,\B)\phi^+_{\A\B}(z)~.
\end{align}
Next, given $\g\in \SL_2(\Z)$ define $\phi_{\g}^-:\HH^-\to \C$ by setting
\begin{align}
\label{eqP2}
	\phi_{\g}^-(z):=B(\g) \phi_{R\g R}^+(-z).
\end{align}
A function $f\in \mathcal{O}(\HH^*)$ is determined by its restrictions $f^\pm:=f|_{\HH^\pm}$ to the upper and lower half-planes. So we may define $f|_k[\gamma,\e]$ for $f\in\mathcal{O}(\HH^*)$ and $k\in \frac12\Z$ and $[\g,\e]\in \widetilde{\SL_2}(\Z)$ by requiring that
\begin{gather}\label{eqn:actionoftildesl2}
	(f|_k[\gamma,\e])^\pm(z)=f^\pm(\gamma z)(\e\phi_\gamma^\pm(z))^{-2k}.
\end{gather}

An extension of this to an action of $\widetilde{\GL_2}(\Z)$ on $\mathcal{O}(\HH^*)$ is now uniquely determined by the choice of action of $\WT R$, because once $f|_k\WT R$ has been suitably defined the action of any other element $\WT \g =[\g,\e]\in \widetilde{\GL_2}(\Z)$ with $\det(\g)=-1$ must satisfy $f|_k\WT\g=(f|_k\WT R)|_k(\WT R^{-1}\WT \g)$, and the action of $\WT R^{-1}\WT\g=[R\g,-A(R,\g)\e]$ has been determined already by (\ref{eqn:actionoftildesl2}) since $\det(R\g)=1$. The choice that we propose is
\begin{gather}\label{eqn:actionofR}
	(f|_k\WT R)(z)=i^{2k}f(-z)
\end{gather}
for $f\in \mathcal{O}(\HH^*)$. We will confirm momentarily (see Proposition \ref{prop:itisanaction}) that this choice leads to a consistent action, and is essentially the unique choice (see Proposition \ref{prop:uniquenessofRaction}) that can be made. 

As we have explained, the rule (\ref{eqn:actionofR}) leads us to define
\begin{gather}\label{eqn:actionofnegdetgamma}
	f|_k[\g,\e] :=i^{2k}(f\circ R)|_k[R\g,-A(R,\g)\e] 
\end{gather}
for $[\gamma,\e]\in \widetilde{\GL_2}(\Z)$ with $\det(\gamma)=-1$. Using this and (\ref{eqn:actionoftildesl2}) we obtain the following explicit prescription for the weight $k$ action of $[\g,\e]\in\widetilde{\GL_2}(\Z)$ on a function $f\in \mathcal{O}(\HH^*)$ for $k\in\frac12\Z$ in terms of $A$, $B$, and the $\phi^+_\g$ for $\g\in\SL_2(\Z)$. 
\begin{gather}\label{eqn:actionoftildegl2summary}
\begin{split}
	(f|_k[\g,\e])^+(z)&:=\begin{cases}
		f^+(\g z)(\e \phi^+_\g(z))^{-2k}&\text{ if }\det(\g)=+1,\\
		f^-(\g z)(i\e A(R,\g)\phi^+_{R\g}(z))^{-2k}&\text{ if }\det(\g)=-1,
	\end{cases}\\
	(f|_k[\g,\e])^-(z)&:=\begin{cases}
		f^-(\g z)(\e B(\g)\phi^+_{R\g R}(-z))^{-2k}&\text{ if }\det(\g)=+1,\\
		f^+(\g z)(i\e A(R,\g)B(R\g)\phi^+_{\g R}(-z))^{-2k}&\text{ if }\det(\g)=-1.
	\end{cases}
\end{split}
\end{gather}

\begin{prop}\label{prop:itisanaction}
The rule~\eqref{eqn:actionoftildegl2summary} defines an action of $\widetilde{\GL_2}(\Z)$ on $\OO(\HH^*)$. 
\end{prop}

Before proving Proposition \ref{prop:itisanaction} we record one more useful identity for the cocycle (\ref{eqAtwist}).
\begin{lem}
\label{lem:AAequalsBBB}
For $\A,\B\in \SL_2(\Z)$ we have $A(\A,\B)A(R\A R,R\B R)= B(\A) B(\B) B(\A\B)$.
\end{lem}

\begin{proof}
Let $\A,\B \in \SL_2(\Z)$ and choose preimages $[\A, \e]$ and $ [\B, \delta]$ in $\WT\SL_2(\Z)$. We compute the conjugate of $[\A, \e] [\B, \delta] $ under $\WT R$ in two different ways. On the one hand we have
\begin{gather}
\begin{split}\label{eqn:Rconjalphabeta1}
\WT R [\A, \e] [\B, \delta] \WT R^{-1} 
&= \WT R[\A \B , A(\A,\B) \e \d  ]\WT R^{-1} \\
&= [R \A \B R, A(\A,\B) B(\A\B) \e \delta ] .
\end{split}
\end{gather}
On the other hand,
\begin{gather}
\begin{split}\label{eqn:Rconjalphabeta2}
\WT R [\A, \e] [\B, \delta] \WT R^{-1} 
&= \WT R [\A, \e] \WT R^{-1} \WT R [\B, \delta] \WT R^{-1} \\
&=  [R\A R, B(\A)\e]  [R\B R, B(\B) \delta]\\
&= [R \A \B R, A(R\A R,R\B R) B(\A)B(\B) \e \delta ] .
\end{split}
\end{gather}
The equality of (\ref{eqn:Rconjalphabeta1}) and (\ref{eqn:Rconjalphabeta2}) proves the claim.
\end{proof}

\begin{proof}[Proof of Proposition \ref{prop:itisanaction}.]
For the sake of completeness we first check that~\eqref{eqn:actionoftildesl2} defines an action of $\widetilde{\SL_2}(\Z)$ on $\OO(\HH^*)$ for $k\in \frac12\Z$. Actually, the action on $\mathcal{O}(\HH^+)$ coincides with that specified by (\ref{eqn:fplustildegamma}) in the sense that $f|_k\iota(\gamma,\phi)=f|_k(\gamma,\phi)$ for $f\in \mathcal{O}(\HH^+)$ and $(\gamma,\phi)\in \widetilde{\SL_2}(\Z)$ by the definition of $\phi_\gamma^+$. To confirm that this defines an action of $\widetilde{\SL_2}(\Z)$ on $\mathcal{O}(\HH^+)$ we just need to check that for $\WT\A=(\A,\phi)$ and $\WT\B=(\B,\psi)$ in $\widetilde{\SL_2}(\Z)$ we have
\begin{align}
\label{eqn:tildegl2action_1}
 (f|_k \WT \A)|_k \WT\B
= 
f|_k (\WT\A\WT\B) 
\end{align} 
for $f\in \OO(\HH^+)$. To see this note that the RHS of~\eqref{eqn:tildegl2action_1} is $f(\A\B z)(\phi(\B z)\psi(z))^{-2k}$ because $\WT\A\WT\B=(\A\B,\phi(\B z)\psi(z))$ by~\eqref{eqn:tildesl2mult}, while the LHS is 
\begin{align}
\begin{split}
	((f|_k(\A,\phi))|_k(\B,\psi))(z)
	&=((f\circ \A) \phi^{-2k} |_k(\B,\psi))(z)\\
	&=f(\A\B z)\phi(\B z)^{-2k}\psi(z)^{-2k}.
\end{split}
\end{align}

Suppose now that $f\in \mathcal{O}(\HH^-)$. With $\WT\A,\WT\B$ as above let $\e,\delta\in \{\pm 1\}$ be determined by requiring that $\iota\WT\A= [\A,\e]$ and $\iota\WT\B=[\B,\delta]$, or equivalently $\phi=\e\phi^+_\A$ and $\psi=\delta\phi^+_\B$. Since $[\A,\e][\B,\delta]=[\A\B,A(\A,\B)\e\delta]$ we require to check that 
\begin{align}
\label{eqn:tildegl2action_2}
 (f|_k [\A,\e])|_k [\B,\delta]
= 
f|_k [\A\B,A(\A,\B)\e\delta] 
\end{align} 
The RHS of (\ref{eqn:tildegl2action_2}) is $f(\A\B z) (A(\A,\B)B(\A\B)\e\delta \phi_{R\A\B R}^+(Rz))^{-2k}$, whereas the LHS of (\ref{eqn:tildegl2action_2}) is
\begin{gather}
	((f\circ \A)(B(\A)\e(\phi^+_{R\A R}\circ R))^{-2k})|_k[\B,\delta]
	=f(\A\B z)(B(\A)\e\phi^+_{R\A R}(R\B z)B(\B)\delta\phi^+_{R\B R}(R z))^{-2k}.
\end{gather}
By (\ref{eqP1}) we have $\phi^+_{R\A R}(R\B z)\phi^+_{R\B R}(R z) =A(R\A R,R\B R)\phi^+_{R\A\B R}(R z)$, so the verification of (\ref{eqn:tildegl2action_2}) reduces to $A(\A,\B)B(\A\B)=A(R\A R, R\B R)B(\A)B(\B)$, which is just the content of Lemma \ref{lem:AAequalsBBB}.

We have verified that (\ref{eqn:actionoftildesl2}) defines an action of $\widetilde{\SL_2}(\Z)$ on $\mathcal{O}(\HH^*)$. We now check that (\ref{eqn:actionofnegdetgamma}) extends this to $\widetilde{\GL_2}(\Z)$. That is, we check that
\begin{gather}\label{eqn:groupactiongl2gl2}
	(f|_k[\A,\e])|_k[\B,\delta] = 
	f|_k[\A\B,A(\A,\B)\e\delta]
\end{gather}
for $f\in \mathcal{O}(\HH^*)$ and $[\A,\e],[\B,\delta]\in \widetilde{\GL_2}(\Z)$. We have already verified this for $\det(\A)=\det(\B)=1$. Suppose that $\det(\A)=-1$ and $\det(\B)=1$. Then for $f\in \mathcal{O}(\HH^*)$, using (\ref{eqn:actionofnegdetgamma}) we have 
\begin{gather}
\begin{split}\label{eqn:fslashdetminusonedetone}
(f|_k[\A,\e])|_k[\B,\delta]&=i^{-2k}((f\circ R)|_k[R\A,A(R,\A)\e])|_k[\B,\delta]\\
&=i^{-2k}(f\circ R)|_k[R\A\B,A(R,\A)A(R\A,\B)\e\delta]
\end{split}
\end{gather}
since $R\A\in\SL_2(\Z)$ and we have already verified that (\ref{eqn:actionoftildesl2}) defines an action of the subgroup $\widetilde{\SL_2}(\Z)<\widetilde{\GL_2}(\Z)$ on $\mathcal{O}(\HH^*)$. On the other hand 
\begin{gather}
f|_k[\A\B,A(\A,\B)\e\delta]=i^{-2k}(f\circ R)|_k[R\A\B,A(R,\A\B)A(\A,\B)\e\delta], 
\end{gather}
which agrees with (\ref{eqn:fslashdetminusonedetone}) because $A(R,\A\B)A(\A,\B)=A(R,\A)A(R\A,\B)$ is a special case of (\ref{eqn:cocycleidentityA}).

To check (\ref{eqn:groupactiongl2gl2}) in case $\det(\A)=1$ and $\det(\B)=-1$ let $f^\pm$ denote the restriction of $f\in\mathcal{O}(\HH^*)$ to $\HH^\pm$. Then on the one hand
\begin{gather}
\begin{split}\label{eqn:fslashdetonedetminusone_1}
((f|_k[\A,\e])|_k[\B,\delta])^\pm 
=(f^\mp\circ \A\B)(\e(\phi_\A^\mp\circ\B))^{-2k}(i\delta A(R,\B) \phi^\pm_{R\B})^{-2k},
\end{split}
\end{gather}
while on the other hand 
\begin{gather}\label{eqn:fslashdetonedetminusone_2}
(f|_k[\A\B,A(\A,\B)\e\delta])^\pm=i^{-2k}(f^\mp\circ \A\B)(A(\A,\B)\e\delta A(R,\A\B)\phi^\pm_{R\A\B})^{-2k}. 
\end{gather}
We have $(\phi^-_\A\circ\B)\phi^+_{R\B}=B(\A)(\phi^+_{R\A R}\circ R\B)\phi^+_{R\B}=B(\A) A(R\A R,R\B)\phi^+_{R\A\B}$. So to verify that (\ref{eqn:fslashdetonedetminusone_1}) and (\ref{eqn:fslashdetonedetminusone_2}) agree on $\HH^+$ it suffices to show that $A(R,\B)  A(R\A R,R\B ) B(\A)=A(\A ,\B) A(R,\A\B)$, or equivalently, 
	$A(R\A R,R\B ) 
	A(\A ,\B) 
	=
	A(R,\A\B)
	A(R,\B) 
	B(\A)$.
We require this identity for arbitrary $\A,\B\in \GL_2(\Z)$ with $\det(\A)=1$ and $\det(\B)=-1$. If we substitute $\B R$ for $\B$ we obtain 
\begin{gather}\label{eqn:AAAABequals1}
	A(R\A R,R\B R) 
	A(\A ,\B R) 
	=
	A(R,\A\B R)
	A(R,\B R) 
	B(\A),
\end{gather}
so we will have confirmed that (\ref{eqn:fslashdetonedetminusone_1}) and (\ref{eqn:fslashdetonedetminusone_2}) agree on $\HH^+$ once we verify (\ref{eqn:AAAABequals1}) for arbitrary $\A,\B\in \SL_2(\Z)$. To achieve this we apply Lemma \ref{lem:AAB} in the form $A(R,\g R)=(-1)A(\g,R)B(\g)$ with $\g=\A\B$ and $\g=\B$, and then apply Lemma \ref{lem:AAequalsBBB} to obtain
\begin{gather}
\begin{split}
	A(R,\A\B R)A(R, \B R) B(\A)&=A(\A\B,R)A(\B,R)B(\A) B(\B) B(\A\B)\\
	&=A(\A\B,R) A(\B,R) A(R\A R,R\B R) A(\A,\B)
\end{split}
\end{gather}
for the right-hand side of (\ref{eqn:AAAABequals1}). So (\ref{eqn:AAAABequals1}) follows from $A(\A,\B) A(\A\B,R) A(\B,R) =A(\A,\B R)$, which is just a special case of the cocycle identity (\ref{eqn:cocycleidentityA}). The verification that (\ref{eqn:fslashdetonedetminusone_1}) and (\ref{eqn:fslashdetonedetminusone_2}) agree on $\HH^-$ is similar and we leave it to the reader.

Finally we consider the case that $\det(\A)=\det(\B)=-1$. Under this assumption we have
\begin{gather}
\begin{split}\label{eqn:fslashdetminusonedetminusone_1}
((f|_k[\A,\e])|_k[\B,\delta])^\pm
=(-1)^{2k}(f^\pm\circ \A\B)(A(R,\A)\e(\phi_{R\A}^\mp\circ \B))^{-2k}(A(R,\B)\delta \phi^\pm_{R\B})^{-2k}
\end{split}
\end{gather}
on the one hand, and 
\begin{gather}\label{eqn:fslashdetminusonedetminusone_2}
(f|_k[\A\B,A(\A,\B)\e\delta])^\pm=(f^\pm\circ \A\B)(A(\A,\B)\e\delta \phi^\pm_{\A\B})^{-2k}
\end{gather} 
on the other. The verification that (\ref{eqn:fslashdetminusonedetminusone_1}) and (\ref{eqn:fslashdetminusonedetminusone_2}) agree is similar to the above so we just present details for the restriction to $\HH^-$. Proceeding as above we observe that 
\begin{gather}
(\phi^+_{R\A}\circ\B)\phi^-_{R\B}=B(R\B)(\phi^+_{R\A }\circ \B R\circ R)(\phi^+_{\B R}\circ R)=B(R\B) A(R\A,\B R)(\phi^+_{R\A\B R}\circ R),
\end{gather}
and conclude from this that (\ref{eqn:fslashdetminusonedetminusone_1}) and (\ref{eqn:fslashdetminusonedetminusone_2}) do agree in the $f^-$ case if it is true that $(-1)A(R,\A) A(R, \B)  A(R\A ,\B  R) B(R \B)=A(\A ,\B) B(\A\B)$ for $\det(\A)=\det(\B)=-1$. Rather than checking this we substitute $\A R$ for $\A$ and $R\B $ for $\B$, and check the resulting identity 
\begin{gather}\label{eqn:AAABequalsAB}
(-1)A(R,\A R) A(R, R\B)  A(R\A  R,R\B  R) B(\B)=A(\A  R,R\B) B(\A\B)
\end{gather}
for $\det(\A)=\det(\B)=1$. By Lemma \ref{lem:AAB} we have $(-1)A(R,\A R)= B(\A)A(\A,R)$. Using this and Lemma \ref{lem:AAequalsBBB} we rewrite (\ref{eqn:AAABequalsAB}) as
\begin{gather}
A(\A,R) A(R, R\B)  A(\A,\B) B(\A\B)=A(\A  R,R\B) B(\A\B),
\end{gather}
and verify it by observing that $A(\A,R)A(\A R, R\B)=A(\A, \B)A(R,R\B)$ is a consequence of the cocycle identity (\ref{eqn:cocycleidentityA}).
\end{proof}

We now verify that the weight $k$ action of $\widetilde{\GL_2}(\Z)$ on $\mathcal{O}(\HH^*)$ that we have defined in (\ref{eqn:actionoftildegl2summary}) is essentially the unique extension of the usual action of $\widetilde{\SL_2}(\Z)$ on $\mathcal{O}(\HH^+)$ such that $\WT R$ acts by mapping $f(z)$ to $\lambda f(-z)$ for some $\lambda\in \C$, for any  $f\in \mathcal{O}(\HH^*)$. 

\begin{prop}\label{prop:uniquenessofRaction}
Suppose that $k\in \frac12\Z$ and $(f,[\gamma,\e])\mapsto f|_k[\gamma,\e]$ is an action of $\widetilde{\GL_2}(\Z)$ on $\mathcal{O}(\HH^*)$ such that the restriction to $\widetilde{\SL_2}(\Z)$ recovers the usual action (\ref{eqn:fplustildegamma}) on $\mathcal{O}(\HH^+)$, and such that $f|_k\WT R=\lambda^{2k}(f\circ R)$ for some $\lambda\in \C$ when $f\in \mathcal{O}(\HH^*)$. Then the action of $\widetilde{\SL_2}(\Z)$ on $\mathcal{O}(\HH^*)$ is given by (\ref{eqn:actionoftildesl2}) and either $\lambda=i$ or $\lambda=-i$.
\end{prop}

\begin{proof}
The first part of the proof of Proposition \ref{prop:itisanaction} shows that the subgroup $\iota(\widetilde{\SL_2}(\Z))\simeq \widetilde{\SL_2}(\Z)$ acts in the usual way on $\mathcal{O}(\HH^+)$, and this is true independently of the choice of $\iota$. However $f|_k\WT R$ is defined, the action of $\widetilde{\SL_2}(\Z)$ on $\mathcal{O}(\HH^-)$ is subsequently determined by 
\begin{gather}
	f|_k[\g,\e]=((f|_k \WT R^{-1})|_k(\WT R[\g,\e]\WT R^{-1}))|_k\WT R
	=((f|_k\WT R^{-1})|_k [R\g R,B(\g)\e])|_k\WT R
\end{gather}
(cf. Lemma \ref{lem:RgammaR}). The reader may check that this leads to (\ref{eqn:actionoftildesl2}) when $f|_k\WT R=(\pm i)^{2k}(f\circ R)$. So suppose that we have an action of $\widetilde{\GL_2}(\Z)$ on $\mathcal{O}(\HH^*)$ such that $f|_k\WT R=\lambda^{2k}(f\circ R)$ for some $\lambda\in \C$. Since $\WT R^2=[I,-1]$ and $(f|_k[I,-1])^+=(-1)^{2k}f^+$ we must have $(-1)^{2k}f^+=((f|_k\WT R)|_k\WT R)^+=\lambda^{4k}(f^+\circ R^2)=\lambda^{4k}f^+$ for $k\in \frac12\Z$. This implies $\lambda=\pm i$, as we claimed.
\end{proof}

%---------------------------------------------------------------------------%
\section{Modular Forms on the Double Half-Plane}
\label{sec:mdlfmsgl2}
%---------------------------------------------------------------------------%

We are now ready to define and characterize modular forms for $\GL_2(\ZZ)$ (i.e. on the double half-plane). For this it is natural to work with vector-valued modular forms so we begin by recalling the definition for $\SL_2(\Z)$. Suppose that $\rho:\WT\SL_2(\Z)\to \GL(V)$ is a (complex) representation of $\WT\SL_2(\Z)$. That is, $V$ is a complex vector space and $\rho$ is a homomorphism of groups. Then for $k\in \frac12\Z$ a function $f:\HH^+\to  V$ is called a {\em (vector-valued) modular form} of weight $k$ for $\WT\SL_2(\Z)$ with representation $\rho$ if $f|_k\WT\g = \rho(\WT \g)f$ for every $\WT\g\in \WT\SL_2(\Z)$. Here $f|_k\WT\g$ is defined just as in (\ref{eqn:fplustildegamma}). That is, $f|_k\WT\gamma:=(f\circ \g)\phi^{-2k}$ for $\WT\g=(\g,\phi)$, or equivalently $f|_k\WT\gamma:=(f\circ \g)(\e\phi^+_\gamma)^{-2k}$ in the notation of (\ref{eqn:actionoftildesl2}), when $\WT\g=[\g,\e]=\iota(\g,\e\phi^+_\g)$. Usually $f$ is assumed to be {\em holomorphic}, meaning that $\lambda\circ f:\HH^+\to \C$ is holomorphic for every linear functional $\lambda:V\to \C$, and some condition is imposed on the growth of $(\lambda\circ f)(it)$ as $t\to\infty$. For example, $f$ is said to be a {\em holomorphic modular form} if $\lambda\circ f$ is holomorphic and $(\lambda\circ f)(it)=O(1)$ as $t\to\infty$ for every linear functional $\lambda$ on $V$, and is called a {\em weakly holomorphic modular form} if $\lambda\circ f$ is holomorphic and there exists $C>0$ such that $(\lambda\circ f)(it)=O(e^{Ct})$ as $t\to\infty$ for every linear functional $\lambda$ on $V$.

To formulate modular forms on the double half-plane we suppose that $\rho:\WT\GL_2(\Z)\to \GL(V)$ is a representation of $\WT\GL_2(\Z)$. Then we call a holomorphic function $f:\HH^*\to V$ a {\em (vector-valued) modular form} of weight $k$ for $\WT\GL_2(\Z)$ with representation $\rho$ if
\begin{gather}
f|_k\WT\g = \rho(\WT \g)f
\end{gather} 
for every $\WT\g\in \WT\GL_2(\Z)$, where the definition of $f|_k\WT\g$ is now given by (\ref{eqn:actionoftildegl2summary}).

Let $M_k(\WT\GL_2(\Z),\rho)$ denote the space of modular forms of weight $k$ for $\WT\GL_2(\Z)$ with representation $\rho$, and interpret $M_k(\WT\SL_2(\Z), \rho)$ analogously in case $\rho$ is a representation of $\WT\SL_2(\Z)$. Our characterization of modular forms for $\WT\GL_2(\Z)$ comes in two parts. To explain the first part suppose that $\rho$ is a representation of $\WT\GL_2(\Z)$. Define $\Res \rho:=\rho|_{\WT\SL_2(\Z)}$ and observe that the restriction $f\mapsto f|_{\HH^+}$ defines a map 
\begin{gather}\label{eqn:resmaponmodularforms}
\Res:M_k(\WT\GL_2(\Z),\rho)\to M_k(\WT\SL_2(\Z),\Res \rho)
\end{gather}
which we also denote by $\Res$. 
The first part of our characterization is that this map is an isomorphism. That is, if $\rho$ is a representation of $\WT\GL_2(\Z)$ then every modular form of weight $k$ and representation $\Res\rho$ for $\WT\SL_2(\Z)$ extends uniquely to a modular form of weight $k$ and representation $\rho$ for $\WT\GL_2(\Z)$, and every modular form of weight $k$ and representation $\rho$ for $\WT\GL_2(\Z)$ arises in this way. 

\begin{thm}\label{thm:resisom}
For $k\in \frac12\Z$ and $\rho:\WT\GL_2(\Z)\to \GL(V)$ a representation the restriction map $\Res:M_k(\WT\GL_2(\Z),\rho)\to M_k(\WT\SL_2(\Z),\Res\rho)$ is an isomorphism.
\end{thm}

Before proving Theorem \ref{thm:resisom} we explain the second part of our characterization, which involves induction, and is necessary because not every representation of $\WT\SL_2(\Z)$ is the restriction of a representation of $\WT\GL_2(\Z)$. Given a representation $\rho:\WT\SL_2(\Z)\to \GL(V)$ let us write $\Ind\rho$ for the representation of $\WT\GL_2(\Z)$ obtained by induction from $\rho$, which we may realize concretely as follows. Let $\Ind V:=V\oplus V$ be the direct sum of two copies of $V$. We regard elements of $\Ind V$ as $2$-component column vectors with first entry corresponding to the first summand and second entry corresponding to the second. Then, given $A,B,C,D\in \GL(V)$ it is natural to write $\left(\begin{smallmatrix}A&B\\ C&B\end{smallmatrix}\right)$ for the element of $\GL(\Ind V)$ that acts as $A$ on the first summand of $\Ind V$, maps the second summand to the first summand via $B$, etc. With this understanding we may define $\Ind\rho:\WT\GL_2(\Z)\to \GL(\Ind V)$ by requiring that
\begin{gather}\label{eqn:indrho}
	\begin{split}
	(\Ind\rho)(\WT\g)&=
	\begin{pmatrix}
	\rho(\WT\g)&0\\
	0&\rho(\WT R\WT\g\WT R^{-1})	
	\end{pmatrix},\\
	\vspace{1cm}\\
	(\Ind\rho)(\WT R\WT\g)&=
	\begin{pmatrix}
	0&\rho(\WT R\WT\g\WT R^{-1})\\
	(-1)^{2k}\rho(\WT\g)&0	
	\end{pmatrix},
	\end{split}
\end{gather}
for all $\WT\g=[\g,\e]$ with $\det(\g)=1$. See Proposition \ref{prop:indrho} for a verification that this rule defines a representation.

We may now seek an analogue of (\ref{eqn:resmaponmodularforms}) for induction but it turns out that the obvious map $M_k(\WT\SL_2(\Z),\rho)\to M_k(\WT\GL_2(\Z),\Ind \rho)$ is generally not an isomorphism; it is injective but not always surjective. 

To remedy this we define an operation $\rho\mapsto\rho^R$ for representations of $\WT\SL_2(\Z)$ by setting
\begin{gather}\label{eqn:rhoR}
\rho^R(\WT\g):=\rho(\WT R\WT \g\WT R^{-1})
\end{gather}
for $\WT\g\in \WT\SL_2(\Z)$. Then the appropriate analogue of (\ref{eqn:resmaponmodularforms}) is a map
\begin{gather}\label{eqn:indmaponmodularforms}
\Ind: M_k(\WT\SL_2(\Z),\rho)\oplus M_k(\WT\SL_2(\Z),\rho^R)\to M_k(\WT\GL_2(\Z),\Ind \rho)
\end{gather}
which we define as follows. Given a representation $\rho:\WT\SL_2(\Z)\to \GL(V)$ we write elements of $M_k(\WT\SL_2(\Z),\rho)\oplus M_k(\WT\SL_2(\Z),\rho^R)$ as pairs $(f,g)$ where $f,g:\HH^+ \to V$ are such that $f\in M_k(\WT\SL_2(\Z),\rho)$ and $g\in M_k(\WT\SL_2(\Z),\rho^R)$. Then for such a pair $(f,g)$ we define $\Ind(f,g):\HH^*\to \Ind V$ by requiring that $\Ind (f,g)^\pm:=\Ind (f,g)|_{\HH^\pm}$ satisfy
\begin{gather}\label{eqn:indf}
\begin{split}
	\Ind (f,g)^+&=
		\begin{pmatrix} f\\ g\end{pmatrix},\\
	\Ind (f,g)^-&=
		\begin{pmatrix} (-i)^{2k}(g\circ R)\\ i^{2k}(f\circ R)\end{pmatrix}.
\end{split}
\end{gather}
We check in Proposition \ref{prop:indf} that $\Ind(f,g)$ belongs to $M_k(\WT\GL_2(\Z),\Ind \rho)$. We can now state the second part of our characterization of modular forms on the double half-plane.
\begin{thm}\label{thm:indisom}
For $k\in \frac12\Z$ and $\rho:\WT\SL_2(\Z)\to \GL(V)$ a representation the induction map $\Ind: M_k(\WT\SL_2(\Z),\rho)\oplus M_k(\WT\SL_2(\Z),\rho^R)\to M_k(\WT\GL_2(\Z),\Ind \rho)$ is an isomorphism.
\end{thm}

Before proving Theorems \ref{thm:resisom} and \ref{thm:indisom} we point out a basic property of representations of $\WT\SL_2(\Z)$ in Lemma \ref{lem:rhoIminusone}, verify that the rule (\ref{eqn:indrho}) defines a representation of $\WT\GL_2(\Z)$ in Proposition \ref{prop:indrho}, and verify that the rule (\ref{eqn:indf}) defines a modular form for $\WT\GL_2(\Z)$ in Proposition \ref{prop:indf}. 
\begin{lem}\label{lem:rhoIminusone}
If $k\in \frac12\Z$ and $\rho:\WT\SL_2(\Z)\to \GL(V)$ is a representation such that $M_k(\WT\SL_2(\Z),\rho)$ is non-zero then $\rho((I,-1))=(-1)^{2k}I$.
\end{lem}
\begin{proof}
For $f:\HH^+\to V$ we have $f|_k(I,-1) =(-1)^{2k}f$ by (\ref{eqn:fplustildegamma}). So if $f$ is not zero then $f|_k(I,-1)=\rho((I,-1))f$ implies $\rho((I,-1))=(-1)^{2k}I$.
\end{proof}

\begin{prop}\label{prop:indrho}
If $\rho$ is a representation of $\WT\SL_2(\Z)$ then the rule (\ref{eqn:indrho}) defines a representation of $\WT\GL_2(\Z)$.
\end{prop}
\begin{proof}
We have to check that $(\Ind\rho)(\WT\A)(\Ind\rho)( \WT\B)=(\Ind \rho)(\WT\A\WT\B)$ for $\WT\A,\WT\B\in \WT\GL_2(\Z)$. Write $\WT\A=[\A,\e]$ and $\WT\B=[\B,\delta]$. If $\det(\A)=\det(\B)=1$ then 
\begin{gather}
\begin{split}
	(\Ind\rho)(\WT\A)(\Ind\rho)( \WT\B)
	&=\begin{pmatrix}
		\rho(\WT \A)&0\\
		0&\rho(\WT R\WT \A \WT R^{-1})
	\end{pmatrix}
	\begin{pmatrix}
	\rho(\WT\B)&0\\0&\rho(\WT R\WT \B \WT R^{-1})
	\end{pmatrix}
	\\
	&=\begin{pmatrix}
		\rho(\WT \A)\rho(\WT\B)&0\\
		0&\rho(\WT R\WT \A \WT R^{-1})\rho(\WT R\WT \B \WT R^{-1})
	\end{pmatrix}
	\\
	&=\begin{pmatrix}
		\rho(\WT \A\WT\B)&0\\
		0&\rho(\WT R\WT \A \WT \B \WT R^{-1})
	\end{pmatrix}
	\\&=(\Ind\rho)(\WT\A\WT\B)
\end{split}
\end{gather}
since $\rho$ is a representation of $\WT\SL_2(\Z)$ by hypothesis. If $\det(\A)=-1$ and $\det(\B)=1$ then $\WT R^{-1}\WT\A=[R\A,-A(R,\A)\e]$ and $\det(R\A)=1$ so we substitute $\WT R^{-1}\WT\A$ for $\WT\g$ in the second line of (\ref{eqn:indrho}) to obtain
\begin{gather}
\begin{split}
	(\Ind\rho)(\WT\A)(\Ind\rho)( \WT\B)
	&=\begin{pmatrix}
		0&\rho(\WT\A\WT R^{-1})\\
		(-1)^{2k}\rho(\WT R^{-1}\WT \A )&0
	\end{pmatrix}
	\begin{pmatrix}
	\rho(\WT\B)&0\\0&\rho(\WT R\WT \B \WT R^{-1})
	\end{pmatrix}
	\\
	&=\begin{pmatrix}
		0&\rho(\WT\A\WT\B\WT R^{-1})\\
		(-1)^{2k}\rho(\WT R^{-1}\WT \A\WT\B )&0
	\end{pmatrix}
\end{split}
\end{gather}
which is what we obtain by substituting $\WT R^{-1}\WT \A\WT\B$ for $\WT\g$ in the second line of (\ref{eqn:indrho}), so is $(\Ind\rho)(\WT\A\WT\B)$ as required. The remaining cases are similar and we leave them to the reader.
\end{proof}

\begin{prop}\label{prop:indf}
If $f$ is a modular form of weight $k$ for $\WT\SL_2(\Z)$ with representation $\rho$ then the rule (\ref{eqn:indf}) defines a modular form of weight $k$ for $\WT\GL_2(\Z)$ with representation $\Ind \rho$.
\end{prop}
\begin{proof}
We require to verify that $\Ind (f,g)|_k\WT\g=(\Ind\rho)(\WT\g)\Ind (f,g)$ for $\WT\g\in \WT\GL_2(\Z)$. It suffices to do this for $\WT\g$ in the generating set $\{\WT R, \WT S, \WT T\}$. For $\WT\g=\WT R$ we observe that
\begin{gather}
\begin{split}
	(\Ind(f,g)|_k\WT R)^+
	&=i^{2k}(\Ind(f,g)^-\circ R)\\
	&=i^{2k}\begin{pmatrix}(-i)^{2k}g\\i^{2k} f\end{pmatrix}\\
	&=\begin{pmatrix}g\\(-1)^{2k} f\end{pmatrix}\\
	&=\begin{pmatrix}0&I\\(-1)^{2k}I&0\end{pmatrix}\begin{pmatrix}f\\g\end{pmatrix}
\end{split}
\end{gather}
which is $((\Ind\rho)(\WT R)\Ind(f,g))^+$ by (\ref{eqn:indrho}). The verification that $(\Ind(f,g)|_k\WT R)^-$ coincides with $((\Ind\rho)(\WT R)\Ind(f,g))^-$ is similar and we omit it. For $\WT\g=\WT S$ or $\WT\g=\WT T$, or indeed for any $\WT\g=[\g,\e]$ with $\det(\g)=1$ we have 
\begin{gather}
\begin{split}
	(\Ind(f,g)|_k\WT \g)^+
	&=\begin{pmatrix}f|_k\WT\g\\ g|_k\WT\g\end{pmatrix}\\
	&=\begin{pmatrix}\rho(\WT\g)f\\\rho^R(\WT\g) g\end{pmatrix}\\
	&=\begin{pmatrix}\rho(\WT\g)&0\\0&\rho^R(\WT\g)\end{pmatrix}\begin{pmatrix}f\\g\end{pmatrix}\\
	&=((\Ind \rho)(\WT\g)\Ind(f,g))^+
\end{split}
\end{gather}
where the second line follows from our hypotheses on $f$ and $g$. The verification that $(\Ind(f,g)|_k\WT \g)^-$ coincides with $((\Ind\rho)(\WT \g)\Ind(f,g))^-$ is again similar so again we leave it to the reader.
\end{proof}

\begin{proof}[Proof of Theorem \ref{thm:resisom}.]
Let $f\in M_k(\widetilde{\GL_2}(\Z),\rho)$. Then $f|_k\WT R =i^{2k}(f\circ R)= \rho(\WT R)f$ so in particular we have $f^-(z)=i^{2k}\rho(\WT R)^{-1}f^+(-z)$ for $z\in \HH^-$ when $f^\pm:=f|_{\HH^\pm}$. So the values of $f$ on $\HH^-$ are determined by those on $\HH^+$ when $f$ is a modular form for $\widetilde{\GL_2}(\Z)$. In other words, the restriction map $M_k(\widetilde{\GL_2}(\Z),\rho)\to M_k(\widetilde{\SL_2}(\Z),\rho)$ is injective. To see that the restriction map is surjective let $f^+\in M_k(\widetilde{\SL_2}(\Z),\rho)$, define a function $f^-:\HH^-\to V$ by setting $f^-(z):=i^{2k}\rho(\WT R)^{-1}f^+(-z)$ for $z\in \HH^-$, and let $f:\HH^*\to V$ be the unique function such that $f|_{\HH^\pm}=f^\pm$. We require to show that $f\in M_k(\widetilde{\GL_2}(\Z),\rho)$. For this it suffices to check that $f|_k\WT R=\rho(\WT R)f$ and $f|_k\WT\g=\rho(\WT \g)f$ for $\WT \g\in \widetilde{\SL_2}(\Z)$. The former of these is true because $(f|_k\WT R)^-=i^{2k}(f^+\circ R)=\rho(\WT R)\rho(\WT R)^{-1}i^{2k}(f^+\circ R)=\rho(\WT R)f^-$, and $(f|_k\WT R)^+=i^{2k}(f^-\circ R)=(-1)^{2k}\rho(\WT R)^{-1}f^+=\rho(\WT R)f^+$, where for the last equality we used $\WT R^{-1}= [I,-1]\WT R$ and Lemma \ref{lem:rhoIminusone}. We have $(f|_k\WT\g)^+=f^+|_k\WT\g=\rho(\WT \g)f^+$ by hypothesis so we just have to check that $(f|_k\WT \g)^-=\rho(\WT \g)f^-$. Write $\WT\g=[\g,\e]$ and note that $f^+|_k\WT R\WT\g\WT R^{-1}=(f^+\circ R\g R)(\e B(\g)\phi^+_{R\g R})^{-2k}$ by Lemma \ref{lem:RgammaR}. Using this and (\ref{eqn:actionoftildegl2summary}) we compute
\begin{gather}
\begin{split}
(f|_k\WT\g)^-&=(f^-\circ\g)(\e B(\g)(\phi^+_{R\g R}\circ R))^{-2k}\\
	&=\rho(\WT R)^{-1}i^{2k}(f^+\circ R\g R\circ R)(\e B(\g)\phi^+_{R\g R}\circ R)^{-2k}\\
	&=\rho(\WT R)^{-1}i^{2k}(f^+|_k(\WT R\WT \g\WT R^{-1})\circ R)\\
	&=\rho(\WT R)^{-1}i^{2k}\rho(\WT R\WT \g\WT R^{-1})(f^+\circ R)\\
	&=\rho(\WT\g)f^-
\end{split}
\end{gather}
which is what we required to show.
\end{proof}

\begin{proof}[Proof of Theorem \ref{thm:indisom}.]
Let $\rho:\widetilde{\SL_2}(\Z)\to \GL(V)$ be a representation and let $F$ be an element of $M_k(\widetilde{\GL_2}(\Z),\Ind \rho)$. Write $\pi_j$ for the projection $\Ind V\to V$ on the $j$-th summand. Then it follows from the definition (\ref{eqn:indrho}) of $\Ind\rho$ that $f:=(\pi_1\circ F)^+$ is an element of $M_k(\widetilde{\SL_2}(\Z),\rho)$ and $g:=(\pi_2\circ F)^+$ is an element of $M_k(\widetilde{\SL_2}(\Z),\rho^R)$. Now the action of $\WT R$ on $\Ind V$ forces $(\pi_1\circ F)^-=(-i)^{2k}(g\circ R)$ and $(\pi_2\circ F)^-=i^{2k}(f\circ R)$. So $F=\Ind (f,g)$. In particular, the induction map (\ref{eqn:indmaponmodularforms}) is surjective. It is also injective because we have just seen that $f=(\pi_1\circ \Ind(f,g))^+$ and $g=(\pi_2\circ\Ind(f,g))^+$. The proof is complete.
\end{proof}

\section{Examples}\label{sec:egs}

We conclude by revisiting the examples discussed in the Introduction, \S\ref{sec:intro}. To say that $f^+$ is a modular form of weight $k$ (in the usual sense) with trivial character is just to say that $f^+\in M_k(\WT\SL_2(\ZZ),1_{\SL})$ where $1_{\SL}:\WT\SL_2(\ZZ)\to \GL(\CC)$ is the trivial representation. Since the trivial representation of $\WT\SL_2(\ZZ)$ trivially extends to the trivial representation $1_{\GL}$ of $\WT\GL_2(\ZZ)$ we can regard Theorem \ref{thm:resisom} as confirming that every $f\in M_k(\WT\GL_2(\ZZ),1_{\GL})$ is obtained by requiring $f^\pm(z)=f^+(\pm z)$ for some (uniquely determined) $f^+\in M_k(\WT\SL_2(\ZZ),1_{\SL})$, where $f^\pm:=f|_{\HH^\pm}$. To see why we restricted to even weights in \S\ref{sec:intro} just note that $M_k(\WT\SL_2(\ZZ),1_{\SL})=\{0\}$ when $k\not\in 2\ZZ$.

To extend the Dedekind eta function to the double half-plane we should apply Theorem \ref{thm:indisom}. To see this let $\rho_\eta:\WT\SL_2(\ZZ)\to \GL(\CC)$ be the character of $\eta$, so that $\eta\in M_{\frac12}(\WT\SL_2(\ZZ),\rho_\eta)$. Then $\rho_\eta$ is not $\Res\rho$ for any representation $\rho:\WT\GL_2(\CC)\to \GL(\CC)$ because 
conjugation by $\WT R$ inverts $\WT T$ and 
$\rho_\eta(\WT T)\neq \rho_\eta(\WT T^{-1})$, but $\GL(\CC)$ is commutative. So the extensions of $\eta$ to the double half-plane are vector-valued with $2$ components, because $\hat\rho_\eta:=\Ind \rho_\eta$ maps $\WT\GL_2(\CC)$ to $\GL_2(\CC)$.
Theorem \ref{thm:indisom} says that $\Ind(\eta,g)$ extends $\eta$ to a modular form on the double half-plane for any $g\in M_{\frac12}(\WT\SL_2(\ZZ),\rho_\eta^R)$. For example, we may take $g=0$. Then $\hat\eta:=\Ind(\eta,0)$ is given explicitly by $\hat\eta(z):=\binom{\eta(z)}{0}$ for $\Im(z)>0$, and $\hat\eta(z):=\binom{0}{i\eta(-z)}$ for $\Im(z)<0$. We have $(\hat\eta|_{\frac12}\WT R)(z)=i\hat\eta(-z)$ according to (\ref{eqn:actionofR}) and
$\hat\rho_\eta(\WT R)=\left(\begin{smallmatrix} 0&1\\-1&0\end{smallmatrix}\right)$ according to (\ref{eqn:indrho}). The reader may check that $\hat\eta|_{\frac12}\WT R = \hat\rho_\eta(\WT R) \hat\eta$, which is just to say that $\hat\eta$ belongs to $M_{\frac12}(\WT\GL_2(\ZZ),\hat\rho_\eta)$, as Theorem \ref{thm:indisom} predicts. The heuristic identity (\ref{eqn:etaRxform}) may now be formulated precisely as 
\begin{gather}
\hat\eta(-z)=\left(\begin{matrix}0&-i\\i&0\end{matrix}\right)\hat\eta(z)
\end{gather}
for $z\in \HH^*$.

\section*{Acknowledgements}

We thank Theo Johnson-Freyd for helpful communication on the double covers of $\GL_2(\ZZ)$, and we thank the anonymous referees for helpful comments.
J.D. gratefully acknowledges financial support from the U.S. National Science Foundation (DMS 1601306). 
D.M. gratefully acknowledges financial support from the Niels Bohr International Academy and from a Carlsberg Distinguished Postdoctoral Fellowship (CF16-0183).

%---------------------------------------------------------------------------%
%---------------------------------------------------------------------------%
%---------------------------------------------------------------------------%
%---------------------------------------------------------------------------%
%---------------------------------------------------------------------------%

\end{document}